\documentclass[11pt]{article}

\usepackage{amsfonts}
\usepackage{amssymb}
\usepackage{amsmath}
\usepackage{amsthm}
\usepackage{graphicx}
\usepackage{color}

\newtheorem{notation}{Notations}[section]
\newtheorem{remarque}[notation]{Remark}
\newtheorem{thm}[notation]{Theorem}

\newtheorem{cor}[notation]{Corollary}
\newtheorem{prop}[notation]{Proposition}
\newtheorem{lem}[notation]{Lemma}

\newcommand{\pE}{\partial^* E}
\newcommand{\hausmoins}{\mathcal{H}^{d-1}}
\newcommand{\ginf}{\|g\|_\infty}
\newcommand{\dive}{\operatorname{div}}

\newcommand{\eps}{\varepsilon}
\newcommand{\R}{\mathbb R}
\newcommand{\C}{\mathcal C}
\newcommand{\Ev}{\widetilde E_v}
\newcommand{\E}{\widetilde E}

\newcommand{\bbar}{\big{|}}
\newcommand{\Z}{\mathbb Z}

\begin{document}

\title{\Large Volume-constrained minimizers for the prescribed curvature problem in periodic media}

\author{
        M. Goldman
        \footnote{CMAP, CNRS UMR 7641, Ecole Polytechnique,
        91128 Palaiseau, France, email: goldman@cmap.polytechnique.fr}
        \and M. Novaga
        \footnote{Dipartimento di Matematica, Universit\`a di Padova,
        via Trieste 63, 35121 Padova, Italy, email: novaga@math.unipd.it}
}
\date{}
\maketitle

\begin{abstract}
\noindent 
We establish existence of compact minimizers of the prescribed mean curvature problem with volume constraint in periodic media. 
As a consequence, we construct compact approximate solutions to the prescribed mean curvature equation. 
We also show convergence after rescaling of the volume-constrained minimizers 
towards a suitable Wulff Shape, when the volume tends to infinity.   
\end{abstract}

\section{Introduction}
In recent years, a lot of attention has been drawn towards the problem of constructing surfaces with prescribed mean curvature.
More precisely, given an assigned function $g:\R^d\to \R$, the problem is finding a hypersurface having mean curvature $\kappa$ satisfying
\begin{equation}\label{kappagi}
\kappa=g.
\end{equation}
To our knowledge, this problem was first posed by S.T. Yau in \cite{yau},
under the additional constraint of the hypersurface being diffeomorphic to a sphere,
and a solution was provided in \cite{tw,huang} when the function 
$g$ satisfies suitable decay conditions at infinity, namely that it decays faster than the mean curvature 
of concentric spheres. Another approach was presented in \cite{bethuel,guida}, by means of  
conformal parametrizations and a clever use of the mountain pass lemma. 
A serious limitation of this method is the impossibility to extend it to dimension higher than three, 
due to the lack of a good equivalent of a conformal parametrization.

Motivated by some homogenization problems in front propagation \cite{noval},
in this paper we look for solutions to \eqref{kappagi} without any topological constraint but with a periodic function $g$,
so that in particular, it does not decay to zero at infinity.
A natural idea is to look for critical points of the prescribed curvature functional
\[
F(E)= P(E)-\int_E g\,dx,
\]
as it is well-known that such critical points solve (\ref{kappagi}), whenever they are smooth \cite{giusti}.
Observe that, in general, it is not possible to construct solutions of (\ref{kappagi}) by a direct minimization of the functional $F$, because such minimizers may not exist or be empty.  

The first result in this setting was obtained by Caffarelli and de la Llave in \cite{cdll} (see also \cite{chambthour}) 
where the authors construct planelike solutions of (\ref{kappagi}) under the assumption that $g$ is small and has zero average, 
by minimizing $F$ among sets with boundary contained in a given strip, 
and then show that the constraint does not affect the curvature of the solution.

Here we are interested instead in compact solutions of (\ref{kappagi}). 
This problem seems difficult in this generality and only some preliminary results,
in the two-dimensional case, 
are presently available \cite{kirsch}. 
However, the following perturbative result has been proved in \cite{noval}: 
given a periodic function $g$ with zero average and small $L^\infty$-norm  
and $\eps$ arbitrarily small, there exists a compact solution of 
\[
\kappa=g_\eps
\]
where $\|g_\eps-g\|_{L^1}\le \eps$. Since the $L^1$-norm does not seem very well suited for this problem, 
a natural question raised in \cite{noval} 
was whether the same result holds when the $L^1$-norm is replaced by the $L^\infty$-norm.

In this paper we answer this question. More precisely, we prove the following result (see Theorem \ref{thmain}):
let $g$ be a periodic H\"older continuous function with zero average on the unit cell $Q=[0,1]^d$ and such that
\begin{equation}\label{condgigi}
\int_E g\,dx\le (1-\Lambda)P(E,Q) \qquad \forall E\subset Q
\end{equation}
for some $\Lambda>0$, where $P(E,Q)$ is the relative perimeter of $E$ in $Q$.
Then for every $\eps>0$ there exist $0< \eps'<\eps$ and a compact solution of
\begin{equation}\label{eqepsp}
\kappa=g+\eps'.
\end{equation}
We observe that \eqref{condgigi} is the same assumption made in \cite{chambthour} in order to prove existence of planelike minimizers. This condition is for instance verified if $||g||_{L^d(Q)}$ is smaller than the isoperimetric constant of $Q$, and allows $g$ to take large negative values.

We construct approximate solutions of \eqref{eqepsp} as volume constrained minimizers of $F$ for big volumes. 
This motivates the study of the isovolumetric function $f:[0,+\infty)\to \R$ defined as
\begin{equation}\label{con}
 f(v)=\min_{|E|=v} F(E).
\end{equation}
As a by-product of our analysis, we are able to characterize the asymptotic shape of minimizers as 
the volume tends to infinity, showing that they converge after appropriate rescaling to the Wulff Shape 
(i.e. the solution of the isoperimetric problem) relative to an anisotropy $\phi_g$ depending on $g$.
We mention that, in the small volume regime, the contribution of $g$ becomes irrelevant and the minimizers 
converge to standard spheres (see \cite{figmag} and references therein).

The plan of the paper is the following:
in Section \ref{secex} we show existence of compact minimizers of \eqref{con}. 
In Section \ref{secprop} we prove that the function $f$ is locally Lipschitz continuous 
and link its derivative to the curvature of the minimizers. 
We also provide an example of a function $f$ which is not differentiable everywhere. 
Let us notice that in these first two parts no assumption is made on the average of $g$ or on its size. 
In Section \ref{secmean} we use the isovolumetric function to find solutions of \eqref{eqepsp}. 
Eventually, in Section \ref{secbig} 
we investigate the behavior of the constrained minimizers of (\ref{con}) as the volume goes to infinity.

\paragraph{Notation and general assumptions.} 
We shall assume that $g$ is a $\C^{0,\alpha}$ periodic function,
with periodicity cell $Q=[0,1]^d$.
We shall also suppose that the dimension of the ambient space is smaller or equal to $7$, 
so that quasi-minimizers of the perimeter have boundary of class $\C^{2,\alpha}$ \cite{giusti}. 
We believe that this restriction is not relevant for the results of this work, but we were not able to remove it. 
For a set of finite perimeter we denote by $P(E)$ its perimeter and by $\pE$ its reduced boundary (see \cite{giusti} for precise definitions). 
Given an open set $\Omega$, we denote by $P(E,\Omega)$ the relative perimeter of $E$ in $\Omega$. 
We take as a convention that the mean curvature (which we define as the sum of all principal curvatures)
of a convex set is positive. If $\nu$ is the outward normal to a set with smooth boundary, 
this amounts to say that the mean curvature $\kappa$ is equal to $\dive(\nu)$.

\paragraph{Acknowledgements.} 
The authors would like to thank Antonin Chambolle for very helpful comments and suggestions. They thank Luca Mugnai for pointing out a mistake in a previous version of  \eqref{alexfench} and G.~Psaradakis for drawing \cite{greco} to their attention.
The first author wish to thank Gilles Thouroude for several discussions on this problem.

\section{Existence of minimizers}\label{secex}
In this section we prove existence of compact volume-constrained minimizers of $F$, 
by showing that for every volume $v$, the problem is equivalent to the unconstrained problem
\begin{equation}\label{uncon}
 \min_{E\subset\R^d} \, F_\mu(E)=\min_{E\subset\R^d} \, P(E)-\int_E g\,dx+\mu \bbar |E|-v \bbar,
\end{equation}
for $\mu>0$ large enough.
We start by studying (\ref{uncon}), showing existence of smooth compact minimizers. 
We then show that there exists $\mu_0$ such that, for $\mu\geq \mu_0$, every compact minimizer of $F_\mu$ has volume $v$. 
In particular, 
this will provide existence of minimizers of (\ref{con}), since $\displaystyle f(v) \leq \min_E \, F_\mu(E)$ for every $\mu\ge 0$.

Denoting by $Q_R$ the cube $[-R/2,R/2]^d$ of sidelength $R$, we consider the spatially constrained problem
\begin{equation}\label{scon}
 \min_{E \subset Q_R} F_\mu(E).
\end{equation}
Having restrained our problem to a bounded domain, we gain compactness of minimizing sequences 
and thus existence of minimizers for (\ref{scon}) by the direct method \cite{giusti}. 
We want to show that these minimizers do not depend on $R$ for $R$ big enough. In order to do so, we need density estimates as \cite{cdll}.

\begin{prop}\label{densite}
There exist two constants $C(d)$ and $\gamma$ depending only on the dimension $d$ such that, if we set $\displaystyle r_0(\mu)=\frac{C(d)}{\mu+\|g\|_\infty}$, then  for every minimizer $E$ of (\ref{scon}) and every $x\in \mathbb{R}^d$,
\begin{itemize}
 \item $|E\cap B_r(x)| \geq \gamma r^d$ for every $r\leq r_0$ if $|B_r(x)\cap E|>0$ for any $r>0$,
\item  $| B_r(x)\backslash E| \geq \gamma r^d$ for every $r\leq r_0$ if $|B_r(x)\backslash E|>0$ for any $r>0$.
\end{itemize}
 \end{prop}

\begin{proof}
 Let $x \in \pE$  then by minimality of $E$ we have
\[P(E)-\int_E g\,dx +\mu \bbar |E| - v \bbar \leq P(E \backslash B_r(x))
-\int_{E\backslash B_r(x)} g \,dx+\mu \bbar |E\backslash B_r(x)| - v\bbar,
\]
hence
\begin{align*}
 P(E) &\leq \int_{E\cap B_r} g\,dx +P(E\backslash B_r)+\mu \bbar |E|-|E\backslash B_r| \bbar\\
	&=  \int_{E\cap B_r} g \,dx+P(E\backslash B_r)+\mu | E \cap B_r |\\
	&\leq |E\cap B_r| (\ginf +\mu)+ P(E\backslash B_r).
\end{align*}

\noindent On the other hand we have 
\[P(E)= \hausmoins(\pE \cap B_r)+\hausmoins(\pE \cap B_r^c)\]
and

\[P(E\backslash B_r)= \hausmoins(E \cap \partial B_r)+\hausmoins(\pE \cap B_r^c).\]

\noindent From these inequalities we get
\[
 \hausmoins(\pE \cap B_r) \leq \hausmoins(E \cap \partial B_r)+(\ginf+\mu) |E \cap B_r|.
\]
Letting $U(r)=|E\cap B_r|$ and using the isoperimetric inequality \cite{giusti}, we have
\begin{align*}
 c(d)U(r)^{\frac{d-1}{d}} & \leq P(E\cap B_r) \\
			&=\hausmoins(\pE \cap B_r)+\hausmoins(\partial B_r \cap E)\\
			& \leq 2 \hausmoins(\partial B_r \cap E)+(\ginf+\mu) U(r).
\end{align*}
 Recalling  that $\hausmoins(\partial B_r \cap E)=U'(r)$ for a.e. $r>0$, we find

\begin{equation}\label{gron}
 c(d)U(r)^{\frac{d-1}{d}}\leq 2 U'(r)+(\ginf+\mu) U(r).
\end{equation}

\noindent The idea is that, when $U$ is small, 
the term $U^{\frac{d-1}{d}}$ dominates the term which is linear in $U$ 
so that we can get rid of it. Letting $\omega_d$ be the volume 
of the unit ball and $\displaystyle r \leq \omega_d^{-\frac{1}{d}}
\left(\frac{c(d)}{2(\mu+\ginf)} \right) $, we then have 
\[
U(r)\leq |B_r|=\omega_d r^d\leq \left(\frac{c(d)}{2(\mu+\ginf)} \right)^{d}.
\]
Raising each side of the inequality to the power $-\frac{1}{d}$ and multiplying by $U$ we get
\[
U(r)^\frac{d-1}{d} \geq \frac{2(\mu+\ginf)}{c(d)} U
\]
and from this 
\[\frac{c(d)}{2} U(r)^\frac{d-1}{d}-(\mu+\ginf) U \geq 0\]
thus finally
\[
c(d) U(r)^\frac{d-1}{d}-(\mu+\ginf) U \geq \frac{c(d)}{2} U(r)^\frac{d-1}{d}.
\]
Putting this back in (\ref{gron}) and letting
$\displaystyle C(d)= c(d)\omega_d^{-\frac{1}{d}}/2$ 
we have
\[
\frac{c(d)}{4} U(r)^\frac{d-1}{d} \leq U'(r)  \qquad \forall r \leq \frac{C(d)}{(\mu+\|g\|_\infty)}. 
\]
If we set $V(r)= U^\frac{1}{d}(r)$ we have 
\[
V'(r)=\frac{1}{d} U'(r) U^{\frac{1-d}{d}}(r)\geq \frac{c(d)}{4d}.
\]

\noindent Integrating we get
\[V(r) \geq \frac{c(d)}{4d} r \qquad \textrm{ hence } \qquad U(r) \geq  \left(\frac{c(d)}{4d}\right)^{d} r^d. \]

The second inequality is obtained by repeating the argument with $E\cup B_r(x)$ instead of $E\backslash B_r(x)$.
\end{proof}

We now estimate the error made by relaxing the constraint on the volume.

\begin{lem}\label{vol}
 For every set of finite perimeter $E$ and every $\mu > \ginf$ we have
\[ \bbar |E|- v\bbar \leq \frac{F_\mu(E)+v\ginf}{\mu-\ginf}.\]
\end{lem}

\begin{proof}
 If $|E| >v$ we have 
\[F_\mu(E)=P(E)-\int_E g+\mu ( |E|-v )\]
thus
\[\mu ( |E|-v )\leq F_\mu(E)+\ginf |E|\]
and from this we find
\[
(\mu-\ginf)(|E|-v)\leq F_\mu(E)+v\ginf.
\]
Dividing by $\mu-\ginf$ we get 
\[
\bbar |E|- v\bbar \leq \frac{F_\mu(E)+v\ginf}{\mu-\ginf}.
\]
If $|E|\le v$ we similarly get
\[
(\mu+\ginf)(|E|-v)\leq F_\mu(E)+v \ginf
\]
hence
\[
\bbar |E|- v \bbar \leq \frac{F_\mu(E)+v\ginf}{\mu+\ginf} \leq \frac{F_\mu(E)+v\ginf}{\mu-\ginf}.
\]
\end{proof}

We now prove that the minimizers do not depend on $R$, for $R$ big enough. Here the periodicity of $g$ is crucial.
\begin{prop}\label{R0}
For every $\mu >\ginf$, there exists $R_0(\mu)$ such that for every $R\geq R_0$, there exists a minimizer  $E_R$ of \eqref{scon} 
verifying $\textrm{diam}(E_R)\leq R_0 $. Equivalently we have 
\[
\min_{E \subset Q_R} F_\mu(E)=\min_{E \subset Q_{R_0}} F_\mu(E)
\]
for all $R\geq R_0$.
\end{prop}

\begin{proof}
Let $E_R$ be a minimizer of \eqref{scon}. Let $Q$ be the unit square and 
\[
N=\sharp \{ z \in \mathbb{Z}^d \; / \;   |\{z+Q\} \cap E_R| \neq 0 \}.
\] 
We want to bound $N$ from above by a constant independent of $R$.

Let $r_0=\frac{C(d)}{\mu+\|g\|_\infty}$ as in Proposition \ref{densite}. 
For all $x \in E_R$ we have
\[|E_R\cap B_r(x)| \geq \gamma r^d \qquad \forall r \leq r_0.\]

\noindent Letting $r_1=\min(r_0 , \frac{1}{2})$, for all $x\in\R^d$ we have 
\[
\sharp \{ z \in \mathbb{Z}^d \; / \;   \{z+Q\} \cap B_{r_1}(x) \neq \emptyset \}\leq 2^d.  
\]
Therefore, we can find at least $N/2^d$ points $x_i$ in $E_R$ such that $B_{r_1}(x_i)\cap B_{r_1}(x_j) 
=\emptyset$ for every $i\neq j$ and such that $x_i \in Q+z_i$ with $|\{z_i+Q\} \cap E_R| \neq 0$ for some $z_i \in \Z$.
 
\noindent We thus have 
\[|E_R| \geq  \sum_i |B_{r_1}(x_i)\cap E_R| \geq \frac{N}{2^d} \gamma r_1^d.\]
This gives us 
\[N\leq \frac{2^d|E_R|}{\gamma r_1^d}. \]

\noindent Letting $B^v$ be a ball of volume $v$, by Lemma \ref{vol} and $F_\mu(E_R)\le F_\mu(B^v)$, we have 
\begin{align*} 
\bbar |E_R|-v \bbar 
&\leq \frac{F_\mu(B^v)+v\ginf}{\mu-\ginf}\\
& \leq \frac{c(d) v^{\frac{d-1}{d}}+2v\ginf}{\mu-\ginf}.
\end{align*}
This shows that 
\[
|E_R| \leq v+\frac{c(d) v^{\frac{d-1}{d}}+2v\ginf}{\mu-\ginf}
\]
so that $N$ is bounded by a constant independent of $R$. 

We now prove that $\textrm{diam}(E_R)\leq C(d) N$. Indeed let $x \in E_R$ and let 
$P_0= [0,1]\times [-R/2,R/2]^{d-1}$ be a slice of $Q_R$ orthogonal to the direction $e_1$. 
For $i \in \mathbb{Z}$ we also set $P_i=P_0+i e_1$. 
Our aim is showing that $E_R$ is contained in a box of size $N$ in the direction $e_1$. 
Up to translation we can suppose that $E_R \cap P_i=\emptyset$ for all $i<0$. 
We want to show that we can choose $\displaystyle E_R \subset \bigcup_{0\leq i \leq N} P_i$.

\begin{figure}[ht]
\centering{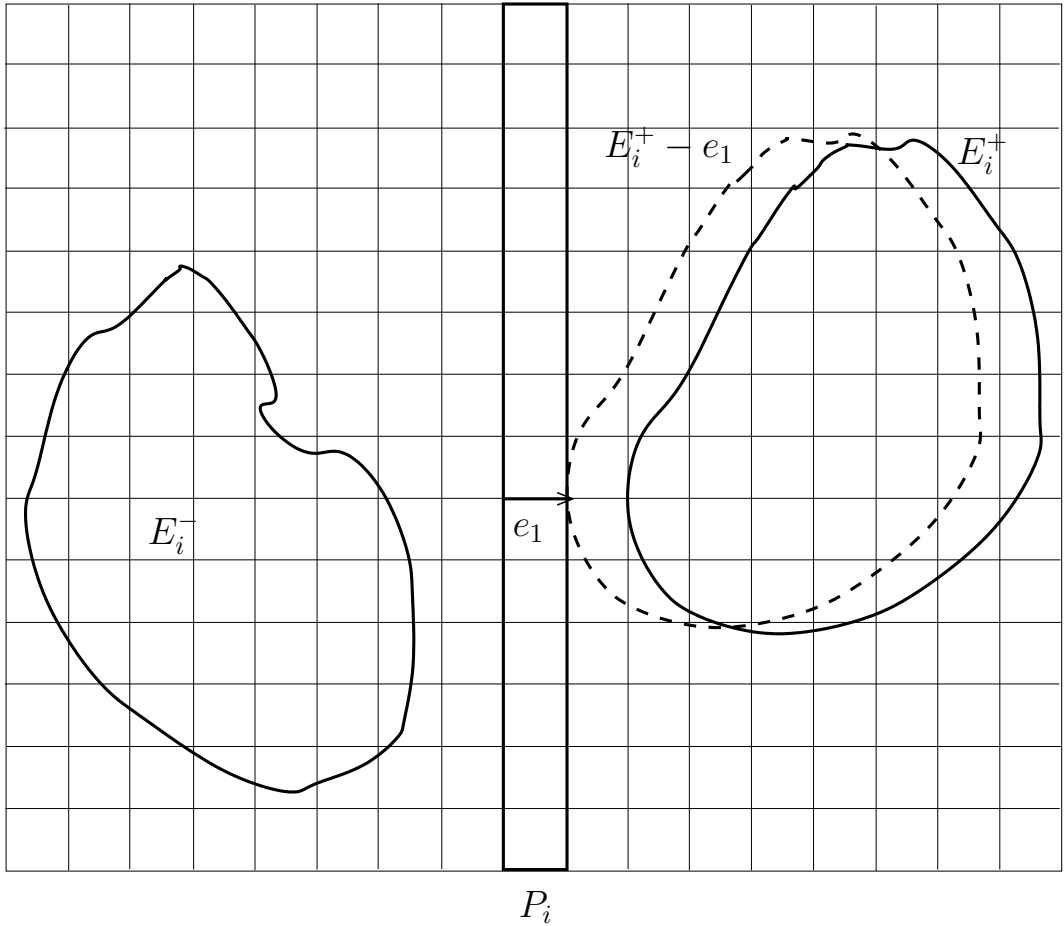}
\caption{the construction in the proof of Proposition \ref{R0}.}
\label{carre}
\end{figure}

\noindent Let $I\leq R $ be the least integer such that $\displaystyle E_R \subset \bigcup_{0\leq i \leq I} P_i$ and suppose $I\geq N$. Because of the definition of $N$, there is at most $N$ slices $P_i$ such that $P_i \cap E_R \neq \emptyset$. Hence there exists $i$ between $0$ and $N$ such that $P_i \cap E_R =\emptyset$. Let $E_i^+=\displaystyle \bigcup_{j>i} E_R \cap P_j$ and $\displaystyle E_i^-=\bigcup_{j<i} E_R \cap P_j$ then if we set $\widetilde{E}_R=E_i^- \cup \{E_i^+-e_1\}$ we have $F_\mu(\widetilde{E}_R)=F_\mu(E_R)$ and $\widetilde{E}_R \subset \bigcup_{0\leq i \leq I-1} P_i$ giving the claim by iterating the procedure (see Figure \ref{carre}).

The same argument applies to any orthonormal direction $e_k$, hence $E_R \subset Q_{2N}$.
\end{proof}

We now prove existence of  minimizers for $F_\mu$.

\begin{prop}\label{proper}
For $\mu >\ginf$, there exists a bounded minimizer of $F_\mu$. 
Moreover such minimizer has boundary of class $\mathcal{C}^{2,\alpha}$,
where $\alpha$ is the H\"older exponent of the function $g$.
\end{prop}

\begin{proof} 
By Proposition \ref{R0} there exists $R_0$ such that $E_R \subset B_{R_0}$ for every $R>0$. Suppose now that there exists $E$ with $F_\mu(E) <F_\mu(E_{R_0})$. Then there exists $\eps >0$ such that
\[ F_\mu(E) +\eps \leq F_\mu(E_{R_0}).\]
Let us show that there exists $R>R_0$ such that 
\[ 
F_\mu(E\cap B_{R}) +\frac{\eps}{2} \leq F_\mu(E_{R_0}).
\]
We start by noticing that $|E\cap B_R|$ tends to $|E|$ and that $\displaystyle \int_{E\cap B_R} g \,dx$ tends to 
$\displaystyle \int_E g\,dx$ when $R\to +\infty$. On the other hand,
\[P(E\cap B_R)=  \hausmoins(E \cap \partial B_R)+\hausmoins(\pE \cap B_R)\]
and we have
\[
\lim_{R\rightarrow + \infty} \hausmoins(\pE \cap B_R)=P(E)
\]
and 
\[
 \lim_{R\rightarrow + \infty} \int_0^R \hausmoins(E \cap \partial B_s) ds= 
 \lim_{R\rightarrow + \infty}|E\cap B_R|=
 |E|. \]
The last equality shows that $\hausmoins(E \cap \partial B_R)$ is integrable so that, for every $R>0$, 
there exists $R'>R$ such that $\hausmoins(E \cap \partial B_{R'})$ is arbitrarily small. 
This implies that we can find a $R$ large enough so that 
\[ F_\mu(E\cap B_{R}) +\frac{\eps}{2} \leq F_\mu(E_{R_0}).\]
The minimality of $E_{R_0}$ yields to a contradiction.\\

We now focus on the regularity. Let E be a minimizer of $F_\mu$ then for every $G$,
\[P(E)-\int_E g\,dx +\mu \bbar |E|-v \bbar \leq P(G) - \int_G g\,dx + \mu \bbar |G|-v \bbar.\]

\noindent Hence
\begin{align*}
 P(E) &\leq P(G)+ \ginf |E \Delta G| + \mu  \bbar |E|-|G| \bbar\\
	& \leq P(G)+(\ginf +\mu) |E \Delta G|.
\end{align*}
$E$ is thus a quasi-minimizer of the perimeter so that, by classical regularity theory \cite{giusti} (see also \cite{morgan2}), 
we get that $\partial E$ is of class $\mathcal{C}^{2,\alpha}$. 
\end{proof}

In order to prove the equivalence between the constrained and unconstrained problems, we will need the following geometric inequality. In the case of convex sets, it directly follows from the
 Alexandrov-Fenchel inequality  (see Schneider \cite{schneider}). For general smooth compact sets with positive mean curvature, it follows from \cite[Cor. 4.6]{greco}.
We include a short proof for the reader's convenience.

\begin{lem}
Let $E$ be a compact set with $\mathcal{C}^2$ boundary and assume that $\kappa>0$ on $\partial E$, where 
$\kappa$ denotes the mean curvature of $\partial E$. Then
\begin{equation}\label{alexfench}
 (d-1)P(E)\ge|E| \,\min_{\partial E} \kappa\,.
\end{equation}
 
\end{lem}
\begin{proof}
Let $\Lambda=\min_{\partial E} \kappa$ then no point of $E$ is at distance of $\partial E$ greater than $\frac{d-1}{\Lambda}$. Indeed, if $x\in E$,
 considering the ball $B(x,R)$ centered in $x$ and of radius $R$, with $R$ the smallest radius such that $\partial E\cap B(x,R)\neq \emptyset$ then $R\leq \frac{d-1}{\Lambda}$ 
since the points of $\partial E\cap B(x,R)$ have curvature less than $\frac{d-1}{R}$.  
Let now $b(x)=dist(x,\R^d\backslash E)$ be the distance function to the complementary of $E$. By the Coarea Formula \cite{AFP}, we have

\[|E|=\int_0^{\frac{d-1}{\Lambda}} P( \{b>t\})\, dt\]
from which we deduce \eqref{alexfench} provided that 
$$
P( \{b>t\})\le P(\{b>0\})=P(E)
$$ 
for a.e. $t>0$. We now prove this inequality.

As $b$ is locally semi-concave in $E$ (see \cite{MM}), that is
$D^2 b\le C\, {\rm Id}$ in the sense of measures,
the singular part of $D^2 b$ is a negative measure. Moreover, letting $Sing$ be the set where $b$ is not differentiable and letting $S=\overline{Sing}$, 
we have that $Sing$ corresponds to the set of points having more than one projection on $\partial E$, 
$b$ is $\mathcal{C}^2$ out of $S$, and $S$ is of zero Lebesgue measure \cite{CM} (and even  $(d-1)$-rectifiable if $\partial E$ is $\mathcal{C}^3$ \cite{MM}).
 The hypothesis that $\partial E$ is $\mathcal{C}^2$ is sharp since there exists sets with $\mathcal{C}^{1,1}$ boundary such that the cut locus is of positive Lebesgue measure \cite{MM}.
The set $S$ is sometime called the {\em cut locus} of $\partial E$.
We refer to \cite{Amb,MM} for a proof of these properties of the distance function $b$. 

If $x\in\{b=t\}$ is a point out of  $S$, by the smoothness of $b$ and by classical formulas there holds \cite{Amb}
\[- \Delta b(x)=\kappa_{\{b=t\}}(x)=\sum_{i=1}^{d-1} \frac{\kappa_i(\pi(x))}{1-b(x)\kappa_i(\pi(x))}\]
where $\pi(x)$ is the (unique) projection of $x$ on $\partial E$ and where $\kappa_i$ are the principal curvatures of $\partial E$. 
By the convexity of the function $\kappa\to \kappa/(1-b\kappa)$, and recalling that the mean curvature of $\partial E$ is positive, 
we get that $\Delta b(x)\le 0$ on $E\backslash S$. 
Finally, since the singular part of the measure $\Delta b$ (which is concentrated on $S$) is non positive, we find that $\Delta b\le 0$ in the sense of measures.

By the Coarea Formula, for a.e. $t>0$ we have  $\hausmoins(\partial \{b>t\}\cap S)=0$, so that for such $t$'s 
\begin{equation}\label{eqpart}
P( \{b>t\})- P(E)=\int_{\{b=t\}} \nabla b \cdot \nu + \int_{\partial E} \nabla b \cdot \nu=\int_{\{0<b<t\}} \Delta b \le 0	\,,
\end{equation}
where $\nu$ denotes the exterior unit normal to the set $\{0<b<t\}$, so that $\nu=-\nabla b$ on $\partial E$
and $\nu=\nabla b$ on $\{b=t\}\setminus S$.

As the vector field $\nabla b$ is bounded and its divergence $\Delta b$ is a Radon measure,
the integration by part formula in \eqref{eqpart} is justified by a result of Chen, Torres and Ziemer \cite[Th. 21.1 (g)]{CTZ}. 
Notice also that, since $\nabla b$ is continuous on $\{b=t\}\setminus S$, 
the (weak) normal trace of $\nabla b$ on $\{b=t\}$ coincides with $\nabla b \cdot \nu$ on $\{b=t\}\setminus S$ \cite[Th. 27.1]{CTZ}. 
\end{proof}
 
\begin{remarque}\rm
 Under the hypothesis $\Lambda:=\min_{\partial E} \kappa>0$, one could also replace \eqref{alexfench} by
\[P(E) R_{\max}\ge |E|\]
where $R_{\max}\le\frac{d-1}{\Lambda}$ is the radius of the largest ball contained in $E$.

\end{remarque}

\begin{remarque}\rm
Notice that the inequality
 \begin{equation}\label{falseineq}
\frac{d-1}{d}\,P(E)^2\ge |E|\int_{\partial E} \kappa
\end{equation}
which  is one of the Alexandrov-Fenchel inequalities (and which implies \eqref{alexfench}) does not hold for a general smooth compact set.  
Indeed, for $d=2$ we can consider a disjoint union of $N$ balls of radius $r_i$, so that the left hand-side is of order $\left(\sum_i r_i \right)^2$ and the
right hand-side is of order $N \left(\sum_i r_i^2\right)$. Hence, if we let $r_i=1/i^2$, 
we get that the left 
hand-side remains bounded while the right hand-side blows-up when the number of balls $N$ increases, thus violating \eqref{falseineq}. 
\end{remarque}

We are finally in position to prove existence of minimizers of problem \eqref{con}. 
\begin{thm}\label{thexist}
Let $d\le 7$, then for all $v>0$ there exists a compact minimizer $E_v$ of (\ref{con}) with $\partial E_v$ of class $\mathcal C^{2,\alpha}$. 
Moreover, $E_v$ is also a minimizer of $F_\mu$ for all 
\begin{equation}\label{eqmmu}
\mu \geq C_1(d) \ginf +C_2(d) v^{-\frac{1}{d}}
\end{equation}
where $C_1(d)$ and $C_2(d)$ are two positive constants depending only on $d$. 
\end{thm}

\begin{proof}
Letting $E_\mu$ be a bounded and smooth minimizer of $F_\mu$, given by Proposition \ref{proper}, 
We will show that $|E_\mu|=v$, for $\mu$ large enough. 
Let $\mu$ be larger than $\ginf$ and suppose by contradiction $|E_\mu| \neq v$. 
Then, if $|E_\mu| >v$, the Euler-Lagrange equation for $F_\mu$ writes
\[ 
\kappa_{E_\mu} = g-\mu 
\]
where $\kappa_{E_\mu}$ is the mean curvature of $E_\mu$. But this is impossible since $\mu> \ginf$, which 
would lead to $\kappa_{E_\mu} <0$, contradicting the compactness of $E_\mu$. 

Thus for $\mu>\ginf$, we have $|E_\mu| <v$ and 
\[ 
\kappa_{E_\mu} = g+\mu .
\]

\noindent Using inequality (\ref{alexfench}) with $E=E_\mu$, and the fact that $|E_\mu| \ge v/2$ by Lemma \ref{vol}, we get
\begin{align*} 
F_\mu(E_\mu)&\geq \frac{1}{d-1} (\mu-\ginf)|E_\mu|-\ginf|E_\mu|\\
	    &\geq \frac{1}{d-1} (\mu-\ginf)\frac{v}{2}-\ginf v.
\end{align*}
On the other hand, $F_\mu(E_\mu)\leq F_\mu(B^v)$, where $B^v$ is a ball of volume $v$, so that
\[C(d) v^\frac{d-1}{d}+\ginf v\geq F_\mu(B^v)\geq  \frac{1}{d-1} (\mu-\ginf)\frac{v}{2}-\ginf v\]
and we finally obtain
\[
\mu \leq C_1(d) \ginf+ C_2(d)v^{-\frac{1}{d}}. 
\]
\end{proof}

\begin{remarque}\rm
The minimizer $E_v$ satisfies the Euler-Lagrange equation
\[
\kappa_E=g+\lambda_v \qquad \textrm{ with } |\lambda_v|\leq \mu,
\]
where $\mu$ verifies \eqref{eqmmu}. In particular, 
$\lambda_v$ and thus also $\|\kappa_E\|_\infty$ are uniformly bounded in $v$, for $v\in [\eps,+\infty)$.

\smallskip

The regularity of $\partial E_v$  also follows from the works of Rigot \cite{rigot} and Xia \cite{xia} 
on quasi-minimizers of the perimeter with a volume constraint.
\end{remarque}

\section{Properties of the isovolumetric function}\label{secprop}
We show here some of the properties of the isovolumetric  $f$ defined by \eqref{con}.

\begin{prop}\label{lipschitz}
The function $f$ is sub-additive and locally Lipschitz continuous. Let $v$ be a point of differentiability of $f$ and $E_v$ be a minimizer of (\ref{con}) then $f'(v)=\lambda_v$ where $\lambda_v$ is the Lagrange multiplier associated to $E_v$, that is,  
$\displaystyle \kappa_{E_v}= g+\lambda_v$. As a consequence, $\lambda_v$ is unique for almost every $v>0$, in the sense that it does not depend on the specific minimizer $E_v$.
\end{prop}

\begin{proof}
Let $E_v$ and $E_{v'}$ be compact minimizers associated to $v$ and $v'$.
Up to a translation we can suppose that $F(E_v \cup E_{v'})=F(E_v)+F(E_{v'})$, so that
\[f(v+v')\leq   F(E_v \cup E_{v'})=F(E_v)+F(E_{v'})=f(v)+f(v')\]
and $f$ is sub-additive.

By Theorem \ref{thexist}, for every $\alpha>0$ there exists $\mu_\alpha$ 
such that, for every $v\geq \alpha$, the constrained problem (\ref{con}) 
and the relaxed one (\ref{uncon}) are equivalent for $\mu\geq \mu_\alpha$. 
Let $v,v' \in [\alpha,+\infty)$, then 
\[f(v)=F(E_v)\leq P(E_{v'})- \int_{E_{v'}} g\,dx + \mu_\alpha |v-v'|=f(v')+\mu_\alpha |v-v'|\]
thus $|f(v)-f(v')|\leq \mu_\alpha|v-v'|$ and $f$ is Lipschitz continuous on $[\alpha,+\infty)$.

We now compute the derivative of $f$. For $v,\eps>0$ we have
\[f(v+\eps)-f(v)\leq F( (1+\eps/v)^{\frac{1}{d}}E_v)-F(E_v).\]

\noindent Let $\delta_\eps=(1+\eps/v)^{\frac{1}{d}} -1$; then $(1+\eps/v)^{\frac{1}{d}}E_v
=E_v+\delta_\eps E_v$. Recalling that $\displaystyle \kappa_{E_v}= g+\lambda_v$ we get
\begin{align*}
P((1+\delta_\eps)E_v)&= P(E_v)+\delta_\eps \int_{\partial E_v} \kappa_{E_v} x \cdot \nu \,d\mathcal{H}^{d-1}
+ o(\delta_\eps)\\
&=P(E_v)+\delta_\eps \int_{\partial E_v} g(x) x \cdot \nu\,d\mathcal{H}^{d-1}
+\delta_\eps \int_{\partial E_v} \lambda_v x \cdot \nu \,d\mathcal{H}^{d-1}
+ o(\delta_\eps)\\
&=P(E_v)+\delta_\eps \int_{\partial E_v} g(x) x \cdot \nu\,d\mathcal{H}^{d-1} +\delta_\eps  \lambda_v d |E_v|+ o(\delta_\eps)
\end{align*}
and
\[
\int_{(1+\delta_\eps)E_v} g =\int_{E_v} g\,dx +\delta_\eps \int_{\partial E_v} g(x)  x \cdot \nu\,d\mathcal{H}^{d-1} +o(\delta_\eps).
\]
From this we obtain
\[F( (1+\eps/v)^{\frac{1}{d}}E_v)-F(E_v)=\delta_\eps v d \lambda_v+ o(\delta_\eps).\]
As $\delta_\eps= \eps/(v d) + o(\eps)$, we find
\begin{align*}
&\limsup_{\eps \rightarrow 0^+} \, \frac{f(v+\eps)-f(v)}{\eps} \leq \lambda_v 
\\
&\liminf_{\eps \rightarrow 0^-} \, \frac{f(v+\eps)-f(v)}{\eps} \geq \lambda_v.
\end{align*}
In particular, if $f$ is differentiable in $v$ we have 
\[f'(v)=\lambda_v .\]
\end{proof}

In fact, the isovolumetric function $f$ is slightly more regular.
\begin{prop}\label{proder}
 Let $\lambda_v^{\max}$ and $\lambda_v^{\min}$ be respectively the bigger and the smaller Lagrange multipliers associated with $v$ then $f$ has left and right derivatives in $v$ and
\begin{equation}\label{eqder}
\lim_{h \to 0^+} \frac{f(v+h)-f(v)}{h}=\lambda_v^{\min} \le \lambda_v^{\max} = 
\lim_{h \to 0^-} \frac{f(v+h)-f(v)}{h}.
\end{equation}
\end{prop}

The proof is based on the following lemma:
\begin{lem}\label{conv}
Let $v_n$ be a sequence converging to $v$. Then there exist sets $E_n$ with $|E_n|=v_n$ and 
\[
f(v_n)= F(E_n),
\]
and a set $E$ with $|E|=v$ and
\[
f(v)=F(E),
\]
such that, up to extraction, $E_n$ tends to $E$ in the $L^1$-topology,  $\partial E_n$ tends to $\partial E$ 
in the Hausdorff sense, and $\lambda_n$ tends to $\lambda$, where $\lambda_n$ (resp. $\lambda$) is the Lagrange multiplier 
corresponding to $E_n$ (resp. to $E$).
\end{lem}

\begin{proof}
By Theorem \ref{thexist}, we can find minimizers $E_n$ of \eqref{con}, with $|E_n|=v_n$. Moreover, 
by Proposition \ref{R0} we can assume that $E_n \subset B_R$ with $R$ independent of $n$.
Since $P(E_n)$ is uniformly bounded from above, it then follows that there exists a (not relabelled) subsequence of $E_n$ 
converging in the $L^1$-topology to a set $E\subset B_R$ with volume $\displaystyle v=\lim_n v_n$.
Moreover, by the lower-semi-continuity of the perimeter and the continuity of $f$, the set $E$ verifies
\[
f(v)=F(E).
\]
Let us now prove that the convergence also occurs in the sense of Hausdorff.

Let $\eps>0$ be fixed and let $x \in E \cap \left\{y \; / \; d(y,\partial E)> \eps \right\}$.
If $x$ is not in $E_n$ then by Proposition \ref{densite} we have
\[
|E_n \Delta E|\geq |B_\eps(x)\backslash E_n|\geq \gamma \eps^d.
\] 
This is impossible if $n$ is big enough because $|E_n\Delta E|$ tends to zero. Similarly, we can show that for $n$ big enough, all the points of $E^c\cap\left\{y \; / \; d(y,\partial E)> \eps \right\}$ are outside $E_n$. This shows that  $\partial E_n \subset \left\{y \; / \; d(y,\partial E)\le \eps \right\}$. Inverting the r\^oles of $E_n$ and $E$, the same argument proves that $\partial E \subset \left\{y \; / \; d(y,\partial E_n)\le \eps \right\}$ giving the Hausdorff convergence of $\partial E_n$ to $\partial E$. Now if $\lambda_n$ is the Lagrange multiplier associated with $E_n$, it is uniformly bounded and we can extract a converging subsequence which converges to some $\lambda\in\R$. 

To conclude the proof we must show that $\kappa_E=g+\lambda$.
As proved for instance in \cite{tamanini}, 
for every $x\in \partial E$ there exists $r>0$ such that for $n$ 
large enough the set $B_r(x)\cap  \partial E_n$ 
is the graph of a function $\varphi_n$, and the set $B_r(x)\cap  \partial E$ is the graph of a function $\varphi$,
in a suitable coordinate system.
We then have that $\varphi_n$ tends uniformly to $\varphi$, as $n\to +\infty$, and 
\begin{equation}\label{eqphin}
-\dive\left(\frac{\nabla \varphi_n}{\sqrt{1+|\nabla \varphi_n|^2}}\right)
=g(x,\varphi_n(x))+\lambda_n
\end{equation}
for all $n$ big enough.
By elliptic regularity \cite{cafcabre}, we can pass to the limit in \eqref{eqphin}
and obtain that $\phi$ solves
\[
-\dive\left(\frac{\nabla \varphi}{\sqrt{1+|\nabla \varphi|^2}}\right)=
\kappa_E=g(x,\varphi(x))+\lambda.
\]
\end{proof}

\noindent {\em Proof of Proposition \ref{proder}.}
\,Let $v>0$ and let 
\begin{equation}\label{liminf}\lambda= \liminf_{\eps\to 0+} f'(v+\eps)
\end{equation}
Notice that, for every $\eps >0$, there exists a $v_\eps \in ]v,v+\eps[$ such that 
\begin{equation}\label{pala} 
f'(v_\eps)\leq \frac{f(v+\eps)-f(v)}{\eps}.
\end{equation}
{}From \eqref{pala} we get
\[
\lambda\leq \liminf_{\eps\to 0+} \frac{f(v+\eps)-f(v)}{\eps}.
\]
Let $\eps_n$ be a sequence realizing the infimum in (\ref{liminf}) and let $E_n\subset B_R$ be a set of volume $v_n=v+\eps_n$ such that
\[
f(v_n)=F(E_n).
\]
By Lemma \ref{conv} the sets $E_n$ converge, up to a subsequence in the $L^1$-topology,
to a limit set $E$, with $|E|=v$ and $\kappa_E=g+\lambda$, where $\displaystyle \lambda=\lim_n\lambda_n$. 
Reasoning as in Proposition \ref{lipschitz}, we see that
\[
\liminf_{\eps\to 0+} \frac{f(v+\eps)-f(v)}{\eps}\geq \lambda 
\geq \limsup_{\eps \rightarrow 0^+} \frac{f(v+\eps)-f(v)}{\eps}
\]
hence $f$ admits a right derivative which is equal to $\lambda_v^{min}$. 
Analogously one can show that $f$ has a left derivative equal to $\lambda_v^{max}$.
\qed

\begin{remarque}\label{rqf} \rm Notice that \eqref{eqder} implies that 
$f$ is differentiable at any local minimum so that, if equation 
\eqref{kappagi} has no solution, either $f$ is increasing on $[0,+\infty)$, 
or there exists $\overline{v}>0$ such that $f$ is increasing on $[0,\overline{v}]$, decreasing on $[\overline{v},+\infty)$,
and is not differentiable at $\overline{v}$.
\end{remarque}



We now give an example of a isovolumetric function $f$ which has a point of nondifferentiability. It is not clear to which extent this is 
a generic phenomenon.

\paragraph{Example.}
Consider a periodic function $g$ which is equal to $0$ everywhere in the unit cell $Q$, 
except in the neighborhood of two points $a$ and $b$. 
Around these points, $g$ is taken to be equal to radial parabolas centered at the point, 
one parabola high and thin, and the other small and large (see Figure \ref{counterexample}).

It is shown in \cite{figmag} that, when the volume $v$ is sufficiently small,
the minimizer $E_v$ is connected. Since the bound on $v$ depends only on $\ginf$, which can be fixed as small as we want, 
we can suppose that the minimizers $E_v$ are connected and are located near $a$ or $b$. 
By the isoperimetric inequality \cite{giusti} we then get that $E_v$ is a 
disk with volume $v$ centered at $a$ or $b$, and will be denoted by $D_v(a)$, $D_v(b)$, respectively.

Therefore, for small volumes the global minimizer is $D_v(a)$ and, once the equality 
\[\displaystyle \int_{D_v(a)} g=\int_{D_v(b)} g\]
is attained, it switches to the disk $D_v(b)$. 
When this transition occurs, there is a jump singularity of the derivative $f'$.

\begin{figure}[ht]
\centering{
\includegraphics[scale=0.5]{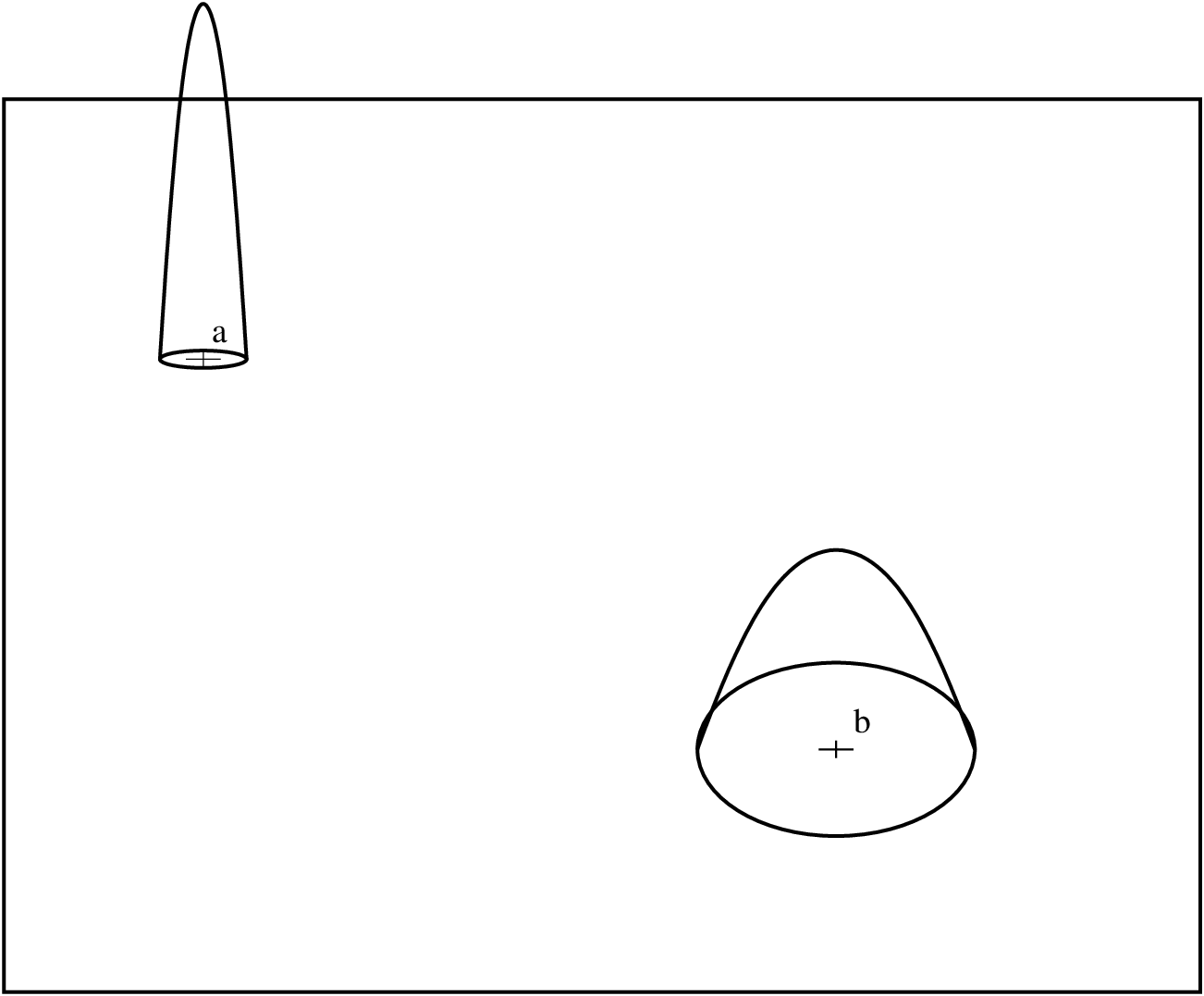}}
\caption{example of a function $f$ with a point of nondifferentiability.}
\label{counterexample}
\end{figure}

\section{Existence of surfaces with prescribed mean curvature}\label{secmean}
In this section we shall assume that $g$ has zero average and satisfies 
\begin{equation}\label{condg}
\int_E g  \le(1-\Lambda) P(E, Q) \qquad \quad \forall E \subset Q
\end{equation}
for some $\Lambda>0$. 
Notice that \eqref{condg} is always satisfied if $\|g\|_{L^d(Q)}$ is small enough, and is precisely the assumption needed in 
\cite{chambthour} (see also \cite{cdll}) to prove existence of planelike minimizers of $F$.
Notice also that, if $g$ satisfies \eqref{condg}, then the inequality in \eqref{condg} holds
for all sets $E\subset\R^d$ of finite perimeter. In particular, this implies the following estimate on 
the function $f$:
\begin{equation}\label{stimaf}
c\,v^\frac{d-1}{d} \le f(v)\le C\, v^\frac{d-1}{d} \qquad {\rm for\ some\ }0<c<C.
\end{equation}

In the sequel we will need a representation result for the functional $F$, due to Bourgain and Brezis \cite{bourgainbrezis}.

\begin{thm}\label{bourbrez}
Let $g$ be a function verifying \eqref{condg} then there exists a periodic and continuous function $\sigma$  with $\max \sigma(x) <1$ satisfying $\dive \sigma=g$. The energy $F$ can thus be written as an anisotropic perimeter:
\[F(E)=\int_{\pE} \left( 1+\sigma(x) \cdot \nu\right).\]
\end{thm}

\noindent Theorem \ref{bourbrez} implies that
\begin{equation}\label{stimF}
 \Lambda P(E)\leq F(E)\leq 2 P(E)
\end{equation}
for all sets $E$ of finite perimeter.


The next Lemma gives an upper bound on the number of ``large'' connected components of 
a volume-constrained minimizer.

\begin{lem}\label{lemnum}
Let $g$ be a periodic $\mathcal{C}^{0,\alpha}$ function with zero average and satisfying \eqref{condg}.
Let $E_v$ be a compact minimizer of \eqref{con}, and let $E_i$ be the connected components of $E_v$.
We can order the sets $E_i$ in such a way that $|E_i|$ is decreasing in $i$.
Given $\delta>0$ let 
\[
N_\delta =\left[ 1+\left(\frac{C}{c}\right)^d\frac{1}{\delta^{d}}\right].
\]
Then
\begin{equation}\label{ndelta}
\sum_{i=N_\delta}^\infty |E_i| \le \delta v.
\end{equation}
\end{lem}

\begin{proof}
Let $x_i = \dfrac{|E_i|}{v}\in [0,1]$. Recalling \eqref{stimaf}, we have
\[
c v^\frac{d-1}{d} \sum_{i=1}^\infty x_i^\frac{d-1}{d}\le \sum_{i=1}^\infty f(|E_i|) = f(v)\le C v^\frac{d-1}{d}, 
\]
hence 
\[
\sum_{i=1}^\infty x_i^\frac{d-1}{d}\le \frac C c
\qquad {\rm and} \qquad \sum_{i=1}^\infty x_i = 1.
\]
Let now $M$ be the smallest integer such that 
\[
\sum_{i=M+1}^\infty x_i < \delta,
\]
we want to prove that $M< N_\delta$. Indeed, we have
\[
\delta\le \sum_{n=M}^\infty x_i = \sum_{n=M}^\infty x_i^\frac{1}{d} x_i^\frac{d-1}{d}
\le x_{M}^\frac{1}{d} \sum_{n=M}^\infty x_i^\frac{d-1}{d}\le \frac C c x_{M}^\frac{1}{d}.
\]
We then obtain
\[
x_M\ge \left(\frac{c}{C}\right)^d\delta^d.
\]
Hence, as 
\[ 
1\ge \sum_{i=1}^M x_i\ge\sum_{i=1}^M x_M=Mx_M,
\] 
by the decreasing property of $x_i$, we get
\[
1\ge M x_M \ge M \left(\frac{c}{C}\right)^d\delta^d,
\]
which gives
\[
M\le \left(\frac{C}{c}\right)^d\frac{1}{\delta^{d}}<N_\delta.
\]
\end{proof}

\subsection{Compact solutions with big volume.}\label{secbig}
{}From \eqref{stimaf}, Proposition \ref{proder} and Remark \ref{rqf},
we immediately obtain the following result.

\begin{prop}\label{progo}
Let $g$ be a periodic $\mathcal{C}^{0,\alpha}$ function of zero average satisfying \eqref{condg}.
Assume that $f'(v)\leq 0$ for some $v>0$. Then there exists $w>0$ such that $f'(w)=0$, therefore 
problem \eqref{kappagi}
admits a compact solution. 
\end{prop}


\begin{thm}\label{thmain}
 Let $g$ be a periodic $\mathcal{C}^{0,\alpha}$ function with zero average and satisfying \eqref{condg}.
There exist $v_n\to +\infty$ and compact minimizers $E_n$ of \eqref{con} such that $|E_n|=v_n$ and 
$E_n$ solves
 \[
 \kappa=g+\lambda_n
 \]
 with $\lambda_n\ge 0$ and $\lambda_n\to 0$ as $n\to +\infty$.
\end{thm}

\begin{proof}
Two situations can occur:

\noindent {\it Case 1.} There exists a sequence $\tilde v_n\to +\infty$ such that $f'(\tilde v_n)\le 0$.
Recalling \eqref{stimaf} we have $f(v)\geq c v^{\frac{d-1}{d}}$, which implies that we can  
find $v_n\ge \tilde v_n$ such that $f$ has a local minimum in $v_n$, hence $\lambda_v=f'(v_n)=0$.

\noindent {\it Case 2.} There exists $v_0>0$ such that $f'(v)>0$ for every $v\ge v_0$. By \eqref{stimaf} we have
$f(v)\leq C v^{\frac{d-1}{d}}$, and
\[
f(v)=f(v_0)+\int_{v_0}^v f'(s) \, ds.
\]
It follows that there exists a sequence $v_n\to +\infty$ such that 
\[\lim_{n\rightarrow +\infty} f'(v_n) = 0.\]
\end{proof}

\begin{cor}\label{cormain}
 Let $g$ be a periodic $\mathcal{C}^{0,\alpha}$ function with zero average and satisfying \eqref{condg}.
 Then for every $\eps>0$ there exists $\eps'\in [0,\eps]$ such that there exists a compact solution of 
 \[
 \kappa=g+\eps'.
 \]
\end{cor}


Notice that for a general function $g$ we cannot let $\eps'=0$ in Corollary \ref{cormain}. 
Indeed, as shown in \cite{BCN}, there are no 
compact solutions to \eqref{kappagi} for periodic functions $g$, of zero average, which are translation invariant 
in some direction and of sufficiently small lipschitz norm.

We expect that condition \eqref{condg} is not necessary for the thesis of Corollary \ref{cormain} to hold,
as suggested by the following result:

\begin{thm}
Let $g$ be a periodic $\mathcal{C}^{0,\alpha}$ function with zero average and such that 
$g|_{\partial Q}=0$.
Then for every $\eps >0$ there exists a compact solution of $$\kappa=g+\eps.$$
\end{thm}

\begin{proof}
Fix $\eps>0$. For $N\in\mathbb N$ we let $E_N$ be a minimizer of the problem
\[
\min_{E\subset Q_N} P(E)-\int_E(g(x)+\eps)\,dx.
\]

Since $g|_{\partial Q}=0$, by strong maximum principle, 
$E_N$ is contained in the interior of $Q_N$ and either $E_N=\emptyset$ or 
$\partial E_N$ is a $\mathcal C^{2,\alpha}$ solution of $\kappa=g+\eps$.

However, from the inequality  
\[
P(E_N) - \int_{E_N} (g(x)+\eps)\,dx
\le P(Q_N) - \eps N^d +  = N^{d-1}\left(2^d-\eps N\right)<0
\]
which holds for all $N>2^d/\eps$, it follows $E_N\ne \emptyset$.
\end{proof}

\subsection{Asymptotic behavior of minimizers.}
For $\eps>0$ and $E\subset\R^d$ of finite perimeter, we let 
\[
F_\eps(E) = \eps^{(d-1)}F\left(\eps^{-1}E\right)
= P(E) - \frac{1}{\eps}\int_{E}g\left( \frac{x}{\eps}\right)\,dx.
\]  
Notice that, given a minimizer $E_v$ of \eqref{con}, 
the set $\eps E_v$ is a volume-constrained minimizer of $F_\eps$.
We recall from \cite[Theorem 2]{chambthour} the following result.

\begin{thm}\label{thconv}
Let $g$ be a periodic $\mathcal{C}^{0,\alpha}$ function with zero average and satisfying \eqref{condg}.
Then there exists a convex positively one-homogeneous function $\phi_g:\R^d\to [0,+\infty)$, 
with $\phi_g(x)>0$ for all $x\ne 0$, such that the functionals $F_\eps$ $\Gamma$-converge,
with respect to the $L^1$-convergence of the characteristic functions, to the anisotropic functional
\[
F_0(E)= \int_{\partial^*E}\phi_g(\nu)\,d\mathcal H^{d-1}
\qquad E\subset\R^d\ {\rm of\ finite\ perimeter.}
\]
\end{thm}

We remark that, with a minor modification of the proof, the result of Theorem \ref{thconv}
also holds if we restrict the functionals $F_\eps$ and $F_0$ to set of prescribed volume.
In particular, by a general property of $\Gamma$-converging sequences \cite{dalmaso},
we have the following consequence of Theorem \ref{thconv}.

\begin{cor}\label{corbel}
Let $\E_\eps$ be minimizers of $F_\eps$ with volume constraint $|\E_\eps|=v$, 
then
\begin{equation}\label{eqineq}
\limsup_{\eps\to 0} F_\eps(\E_\eps)\le \min_{|\E|=v}F_0(\E).
\end{equation}
Moreover, if $|\E_\eps\Delta \E|\to 0$ for some $\E\subset\R^d$, as $\eps \to 0$, 
then $|\E|=v$ and $\E$ is a volume-constrained minimizer of $F_0$.
More generally, if $\E_\eps \to \E$ in the $L^1_{\rm loc}$ topology, then $\E$ is a minimizer of $F_0$ 
with volume constraint $|\E|\le v$. 
\end{cor}
 
\noindent Given the function $\phi_g$ as above, we let 
\[
W_g = \left\{ x\in\R^d:\ \max_{\phi_g(y)\le 1}x\cdot y \le 1\right\}
\]
be the Wulff Shape corresponding to $\phi_g$. It is well-known 
that $W_g$ is the unique minimizer of $F_0$ with volume constraint, 
up to homothety and translation \cite{wulff, taylor}.

By Theorem \ref{thconv} we can characterize the asymptotic shape 
of the constrained minimizers as the volume tend to infinity.

\begin{thm}\label{thasym}
Let $d\le7$. For $v>0$ we let $E_v$ be volume-constrained minimizers of \eqref{con},
whose existence is guaranteed by Theorem \ref{thexist}. 
Then, there exist points $z_v\in \R^d$ such that letting 
\[
\widetilde E_v = \left(\frac{|W_g|}{v}\right)^\frac 1 d\,E_v + z_v
\]
it holds
\begin{equation}\label{eqwulff}
\lim_{v\to +\infty} \left| \widetilde E_v \Delta W_g\right| = 0.
\end{equation}
\end{thm}

\begin{proof}
Notice first that $\widetilde E_v$ is a minimizer of $\displaystyle F_{(\frac{|W_g|}{v})^\frac 1 d}$,
with volume constraint $|\widetilde E_v|=|W_g|$. Moreover, by \eqref{stimaf}
the perimeter of $\widetilde E_v$ is uniformly bounded in $v$.\\

\noindent {\em Case 1.} Let us consider the case $d=2$.
Assume first that $\Ev$ is connected. Then we have
\[
{\rm diam}(\Ev)\le P(\Ev)/\pi,
\]
hence the sets $\Ev$ are all contained, up to a translation, in a fixed ball centered in 
the origin. By the compactness theorem for sets of finite perimeter \cite{giusti},
there exist a bounded set $\E_\infty$ of finite perimeter and a sequence $v_k\to\infty$ 
such that $|\E_\infty|=|W_g|$ and
\begin{equation*}
\lim_{k\to +\infty} \left| \widetilde E_{v_k}\Delta \E_\infty\right| = 0.
\end{equation*}
Since by Theorem \ref{thconv} the set $\E_\infty$ is also a volume-constrained minimizer of $F_0$,
by uniqueness of the minimizer it follows that $\E_\infty$ is equal to $W_g$ up to a translation.

We now consider the general case when the sets $\E_v$ are not necessarily connected.
In particular we can write $\Ev=\cup_{i\ge 1} \Ev^i$, with $|\Ev^i|$ a decreasing sequence and
$\sum_{i\ge 1}|\E_v^i|=1$. Reasoning as before, there exists a sequence $v_k\to +\infty$ such 
that for all $i\in\mathbb N$ the sets $\E_{v_k}^i$ converge to $\rho_i W_g$, up to a translation,
where $\rho_i\in [0,1]$ is a decreasing sequence. 
Moreover, by Lemma \ref{lemnum}, for all $\delta>0$ there exists $N_\delta\in\mathbb N$ such that 
$\sum_{i=N_\delta}^\infty |\Ev^i|\le \delta |W_g|$ for all $\delta>0$,
which implies in the limit 
\begin{equation}\label{rouno}
\sum_{i=1}^\infty \rho_i^2 = 1.
\end{equation}
We claim that $\rho_1=1$ and $\rho_i=0$ for all $i>1$. Indeed, 
from \eqref{eqineq} we have
\[
F_0(W_g)\ge \limsup_{k\to +\infty}F_{\left(\frac{|W_g|}{v_k}\right)^\frac 1 2}(\E_{v_k})\ge \sum_{i=1}^{+\infty} F_0(\rho_i W_g)
=F_0(W_g)\sum_{i=1}^{+\infty} \rho_i\,.
\]
Recalling \eqref{rouno}, this implies
\[
\sum_{i=1}^{+\infty} \rho_i=\sum_{i=1}^{+\infty} \rho_i^2=1
\]
which proves the claim.

\noindent {\em Case 2.} We now turn to the general case. 
Let $v_k \to +\infty$ and let $\eps_k= \left(|W_g|/v_k\right)^\frac 1 d$. For all $k$, let  
$\{Q_{i,k}\}_{i\in\mathbb N}$ be a partition of  $\R^d$ into disjoint cubes of equal volume larger than $2|W_g|$, such that 
the sets $\E_{v_k}\cap Q_{i,k}$ are of decreasing measure, and 
let  $x_{i,k}=|\E_{v_k}\cap Q_{i,k}|/|W_g|$. By the isoperimetric inequality \cite{giusti}, 
there exist $0<c<C$ such that 
\begin{align*}
c\sum_i x_{i,k}^\frac{d-1}{d}&=c\sum_i \min\left(\frac{|\E_{v_k}\cap Q_{i,k}|}{|W_g|},\frac{|Q_{i,k} \backslash \E_{v_k}|}{|W_g|}\right)^\frac{d-1}{d}\\
			   & \le \sum_i P(\E_{v_k},Q_{i,k})\\			   
   			   &\le \sum_i \frac{1}{\Lambda} \int_{\partial \E_{v_k} \cap Q_{i,k}} 
   			   \left(1+ \sigma\left(\frac{x}{\eps_k}\right)\cdot \nu\right)\,d\mathcal H^{d-1}\\
			   &\le \frac{1}{\Lambda} F_{\eps_k}(\E_{v_k})\le C
\end{align*}
hence
\[
\sum_{i=1}^{+\infty} x_{i,k}=1 \qquad 
{\rm and}\qquad 
\sum_{i=1}^{+\infty} x_{i,k}^\frac{d-1}{d}\le \frac{C}{c}.
\]
Reasoning as in Lemma \ref{lemnum} we obtain that 
for all $\delta>0$ there exists $N_\delta\in\mathbb N$ such that
\begin{equation}\label{ndelta2}
\sum_{i=N_\delta}^\infty x_{i,k} \le \delta.
\end{equation}
Up to extracting a subsequence, we can suppose that  
$x_{i,k}\to \alpha_i^d\in [0,1]$ as $k\to +\infty$ for every $i\in\mathbb N$, so that by \eqref{ndelta2} we have
\begin{equation}\label{eqstella}
\sum_i \alpha_i^d=1.
\end{equation}
Let $z_{i,k} \in Q_{i,k}$. Up to extracting a further subsequence, we can suppose that $d(z_{i,k},z_{j,k}) \to c_{ij} \in [0,+\infty]$, and
\[
\left(\E_{v_k}-z_{i,k}\right) \to E_i \quad \textrm{ in the } L^1_{\rm loc}\textrm{-convergence}
\]
for every $i\in\mathbb N$ (see Figure \ref{masssplit}). 
By Corollary \ref{corbel} we  thus have 
\[
E_i= \rho_i W_g \qquad \rho_i\in [0,1].
\]

We say that $i \sim j$ if $c_{ij} < +\infty$ and we denote by $[i]$ the equivalence class of $i$.
Notice that $E_i$ equals $E_j$ up to a traslation, if $i\sim j$. We want to prove that
\begin{equation}\label{rodi}
\sum_{[i]} \rho_i^{d}\ge 1,
\end{equation}
where the sum is taken over all equivalence classes.
For all $R>0$ let $Q_R=[-R/2,R/2]^d$ be the cube of sidelength $R$. Then for every $i\in\mathbb N$,
\[
|E_i| \geq |E_i \cap Q_R|= \lim_{k\to +\infty} \left|\left(\E_{v_k} -z_{i,k}\right) \cap Q_R\right|.
\]
If $j$ is such that $j \sim i$ and $c_{ij} \le \frac{R}{2}$, possibly increasing $R$ we have 
$Q_{j,k}- z_{i,k} \subset Q_R$ for all $k\in\mathbb N$, so that
\[
\lim_{k\to +\infty} \left|\left(\E_{v_k} -z_{i,k}\right) \cap Q_R\right|\geq \lim_{k \to +\infty} 
\sum_{c_{ij} \leq \frac{R}{2}} |\E_{v_k} \cap Q_{j,k}|=\sum_{c_{ij} \leq \frac{R}{2}} \alpha_j^{d} |W_g|.
\]
Letting $R\to +\infty$ we then have
\[
|E_i| \geq \sum_{i\sim j} \alpha_j^d |W_g|
\]
hence, recalling \eqref{eqstella},
\[
\sum_{[i]} |E_i| \ge |W_g|,
\]
thus proving \eqref{rodi}.

\begin{figure}[ht]
\centering{\input{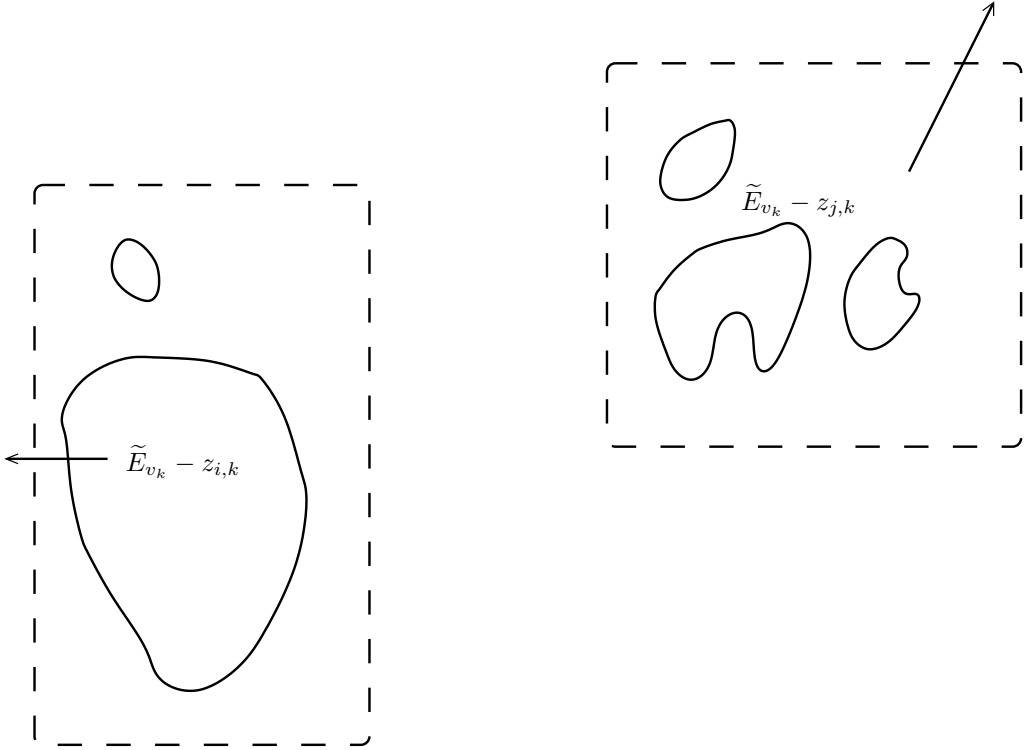}}
\caption{the construction in the proof of Theorem \ref{thasym}.}
\label{masssplit}
\end{figure}

Let us now show that 
\begin{equation}\label{rofinal}
 \sum_{[i]} \rho_i^{d-1} =1.
\end{equation}
Up to passing to a subsequence, from now on 
we shall assume that $c_{ij}=+\infty$ for all $i\ne j$.
Let $I \in \mathbb{N}$ be fixed. Then for every $R>0$ there exists $K \in \mathbb{N}$ such that for every $k\ge K$ and $i$, $j$ less than $I$, we have 
\[d(z_{i,k},z_{j,k}) > R. 
\]
For $k \ge K$ we thus have
\begin{align*}
 F_{\eps_k}(\E_{v_k})\geq & \sum_{i=1}^I \int_{\partial \E_{v_k} \cap (B_R+z_{i,k})}
  \left(1+ \sigma\left(\frac{x}{\eps_k}\right) \cdot \nu\right)\,d\mathcal H^{d-1}
 \\
		=   & \sum_{i=1}^I \int_{\partial (\E_{v_k}-z_{i,k}) \cap B_R} \left(1+ \sigma\left(\frac{x}{\eps_k}\right) \cdot \nu\right)\,d\mathcal H^{d-1}
		\\
		=   & \sum_{i=1}^I F_{\eps_k}(\E_{v_k} -z_{i,k} ,B_R)
\end{align*}
where 
$$
F_\eps(E,B_R)=\int_{\partial E \cap B_R} \left(1+ \sigma\left(\frac{x}{\eps_k}\right) \cdot \nu\right)\,d\mathcal H^{d-1} .
$$ 
{}From this, \eqref{eqineq} and the $\Gamma$-convergence of $F_\eps(\cdot, B_R)$ to $F_0(\cdot,B_R)$, we get
\[
F_0(W_g)\ge \limsup_{\eps_k\to 0}F_{\eps_k}(\E_{v_k})\ge \sum_{i=1}^I \liminf_{\eps_k \to 0} F_{\eps_k}(\E_{v_k} -z_{i,k} ,B_R) \ge \sum_{i=1}^I F_0(E_i, B_R).
\]
For $R>{\rm diam}(W_g)$ we have  $F_0(E_i, B_R)=F_0(E_i)$ because $E_i= \rho_i W_g$  and therefore
\[
F_0(W_g)\ge\sum_{i=1}^I F_0(E_i)=\sum_{i=1}^I \rho_i^{d-1} F_0(W_g).
\]
Letting $I \to +\infty$ we get \eqref{rofinal}. 

Recalling \eqref{rodi}, from \eqref{rofinal} we then obtain
\[
\sum_{i}\rho_i^{d-1}=\sum_{i}\rho_i^{d}=1.
\]
As before, this implies $\rho_1=1$ and $\rho_i=0$ for all $i>1$, thus giving
\[
\lim_{k\to +\infty}
\left|\left(\E_{v_k}-z_{1,k} \right) \Delta W_g\right|=0.
\]
By the uniqueness of the limit this shows that the whole sequence $\E_v$ tends to $W_g$ as $v\to +\infty$, 
up to suitable translations.
\end{proof}

\begin{remarque}\rm
Let us point out that, if uniform density estimates for $\E_v$ were available, we would get Hausdorff 
convergence instead of $L^1$ convergence in \eqref{eqwulff}, 
showing in particular that the sets $\E_v$ are connected for $v$ large enough (see \cite{morganros}). 
We believe that such estimates are true even if we were not able to prove them. 
\end{remarque}

\begin{remarque}\rm 
The asymptotic behavior of minimizers of \eqref{con}, in the small volume regime, have been considered in \cite{figmag},
where the authors prove a result similar to Theorem \ref{thasym}, with the Wulff Shape $W_g$  replaced by 
the Euclidean ball, showing in particular that the volume term becomes irrelevant for small volumes.
\end{remarque}

\begin{remarque}\rm
Notice that the results of this paper can be extended with minor modifications of the proofs 
to anisotropic perimeters of the form
\[
P_\phi(E)=\int_{\pE} \phi(\nu) d\mathcal{H}^{d-1}
\]
where $\phi:\R^d\to [0,+\infty)$ is a smooth and uniformly convex norm on $\R^d$, with $d\le 3$ \cite{ASS}.
\end{remarque}


\end{document}